\documentclass{amsart}

\usepackage{setspace,amsfonts,amsthm,amssymb,amsmath,comment}
\usepackage[mathscr]{euscript}
\usepackage[utf8]{inputenc}
\usepackage[T1]{fontenc}

\newtheorem{Th}{Theorem}
\newtheorem*{Th*}{Theorem}

\newtheorem{Wn}{Corollary}
\newtheorem{Le}{Lemma}

\theoremstyle{definition}

\newtheorem{Prz}{Example}

\newtheorem{Uw}{Remark}

\newcommand\mR{{\mathbb R}}

\newcommand\mN{{\mathbb N}}
\newcommand\mQ{{\mathbb Q}}

\newcommand\restr{\mathord{\restriction}}

\newcommand\B{{\mathscr B}}

\newcommand\PC{{\mathscr{P}}}

\newcommand\cn{\colon}

\newcommand\wi{\geqslant}
\newcommand\mn{\leqslant}

\newcommand\ol{\overline}

\newcommand\fs{F_\sigma}
\newcommand\gd{G_\delta}
\newcommand\fcb{\mathfrak{fcBor}}
\newcommand\flb{\mathfrak{flB}}
\newcommand\fii{{\varphi}}

\title{On Borsuk's non-retract theorem}
\author{Waldemar Sieg}
\address{Kazimierz Wielki University, Institute of Mathematics, Powsta\'nc\'ow Wielkopolskich 2, 85-090 Bydgoszcz, Poland}
\email{waldeks@ukw.edu.pl}


\begin{document}

\maketitle

\begin{abstract}
Let $X$ be the Hausdorff topological space and let $A$ be both an $\fs$ and $\gd$ subset of $X$. Let also $f\cn A\to\mR$ be a function
for which the inverse image of every open subset $U\subset\mR$ is $\fs$ in $A$. We will prove that there is a linear extension operator 
$\fii^\star$ such that $\fii^\star(f)$ has the same property on $X$. An analogous result is proved for the first class function defined on an 
analogous subset of $\mR$. We will also show that the extension map is (with a supremum norm) an isometry. 
In the second part of the paper we deal with the classical Borsuk's non-retract theorem. It says that a unit sphere in $\mR^n$ is not 
a continuous retract of the unit closed ball. We will show that such a unit sphere is a piecewise continuous retract of the unit closed ball.
\end{abstract}

\section{Introduction}

Let $X$ and $Y$ be the Hausdorff topological spaces, let $A$ be an $\fs$ and $\gd$ subset of $X$ and let $f\cn A\to\mR$ be a function
for which the inverse image of every open subset $U\subset\mR$ is $\fs$ in $A$.
The mapping $r\cn X\to A$ is an {\em algebraic retraction} if $r(x)=x$ for every $x\in A$. We say that a function $f\cn X\to\mR$ is
{\em piecewise continuous}, if there is an increasing sequence $\left(X_n\right)\subset X$ of nonempty closed sets such
that $X=\bigcup^\infty_{n=0}X_n$ and the restriction $f_{\restr X_n}$ is continuous for each $n\in\mN$.
The family of all piecewise continuous functions $X\to\mR$ is denoted by $\PC(X)$. 

\begin{Prz}
Let $E(t)$ be an \emph{entier} of a real number $t$, $A=[0,1)\subset\mR$ and let $g\cn\mR\to A$ be a mapping given by the formula $g(x)=x-E(x)$.
Since 
$$
[0,1)=\bigcup_{k=1}^\infty\left[0,1-\frac{1}{k}\right]
$$ and 
$$
\mR=\bigcup_{n=-\infty}^\infty\left(n+[0,1)\right)=\bigcup_{n=-\infty}^\infty\left[n,n+1\right),
$$
$g$ is a piecewise continuous retraction of $\mR$ onto interval $A$.
\end{Prz}
Recall that $\B_1(X,Y)$ is the space of the pointwise limits of sequences of conti\-nuous functions defined on $X$. Moreover, 
$\B_1^\star(X,Y)\subsetneq\B_1(X,Y)$ is the space of the pointwise limits of those sequences
of continuous functions $(f_n)$ which are constant for every $n>n_\epsilon$ with $n_\epsilon\in\mN$. In \cite{Kirch} it is shown that the inclusion 
$\B_1^\star\subset\PC$ holds for metric spaces and the inclusion $\PC\subset\B_1^\star$ holds for all complete metric spaces.  
We say that a function $f\cn X\to Y$ is
\begin{itemize}
\item {\em of the first Borel class}, if the $f$-inverse image of every open subset of $Y$ is $\fs$ in $X$ (see \cite[p. 1]{JayRog});
\item {\em of the first level Borel class}, if the $f$-inverse image of every $\fs$ subset of $Y$ is $\fs$ in $X$ (see \cite[p. 1]{JayRog}).
\end{itemize}
Note that in the last definition, the words ''$\fs$ subset of $Y$'' may be equi\-valently replaced by ''closed subset of $Y$''.
The symbols $\fcb(X,Y)$ and $\flb(X,Y)$ stand for the families of the first Borel and first level Borel classes of functions
defined on $X$, respectively. The inclusion $\flb(X,Y)\subset\fcb(X,Y)$
holds whenever any open set in $Y$ is $\fs$ (for example if $Y$ is perfectly normal). A function witnessing the fact that
$\flb(\mR,\mR)\neq\fcb(\mR,\mR)$ is the classical Riemann function (i.e. $f(x)=0$ for $x$ irrational and $f(x)=\frac{1}{q}$ if $x=\frac{p}{q}$ where
$p$, $q$ are mutually prime) which is of the first Borel class but not of the first level Borel class.
If $f\in\fcb^b(A,Y)$ and $f\in\B_1^b(A,Y)$ are bounded elements of $\fcb(A,Y)$ and $\B_1(A,Y)$, respectively, the \emph{supremum norm}
$\|\cdot\|_A$ of $f$ on $A$ is defined as $\|f\|_A=\sup_{x\in A}|f(x)|$.

In the first part of the paper we deal with the general problem in the theory of real functions, which is inspired by the Tietze extension theorem:
\begin{center}\emph{$(P)$ Let $A$ be a nonempty subset of a topological space $X$ and let $f_0\in\mR^A$ be a function with a certain
property $(W)$. Can $f_0$ be extended to a function $f\in\mR^X$ with the same property $(W)$?}\end{center}
Classical Tietze-Urysohn's theorem says that every continuous map defined on a closed subset of a normal topological space can be extended
to a function continuous on all that space. In 1982, Jayne and Rogers \cite{JayRog} proved the following
\begin{Th}\label{Jay_Rog}
Let $X$ and $Y$ be metric spaces, $Y$ complete. Let $A\subset X$ and $f\in\PC(A,Y)$. Then there exists $A_1\subset X$ which is intersection of a $\gd$ and
an $\fs$ sets and a function $\ol{f}\in\PC(A_1,Y)$ extending $f$.
\end{Th}
Finally, in 2008 Karlova \cite{Karlova} showed, that if $X$ is perfectly paracompact and $A$ is its complete metrizable subspace, then every
continuous function $f\cn A\to Y$ has an extension $\ol{f}\in\fcb(X)$.

In this paper we prove an analogous results for the classes of first Borel and Baire one functions. In (key) Lemma $\ref{funkcja fi}$ we construct an
algebraic retraction $\fii$ whose properties allow to formulate and prove these results.
\section{Properties and useful lemmas}
The first Lemma is a well-known fact and follows for example from \cite[Th. 2, Section VI of $\S$31]{Kur}.
\begin{Le}\label{suma_fcb}
Let $X$ be a topological space. If $f,g\in\fcb(X,\mR)$ then \\$f+g\in\fcb(X,\mR)$.
\end{Le}
Since $|f|$ is a composition of $f$ and a continuous mapping $x\mapsto|x|$, we have
\begin{Le}\label{modul_fcb}
Let $X$ be a topological space. If $f\in\fcb(X,\mR)$ then \\$|f|\in\fcb(X,\mR)$.
\end{Le}
From Lemmas \ref{suma_fcb} and \ref{modul_fcb} it follows that
\begin{Wn}\label{krata_liniowa}
Let $X$ be a topological space. Then $\fcb(X,\mR)$ is a linear lattice.
\end{Wn}
Next lemma is crucial for the main result.
\begin{Le}\label{funkcja fi}
Let $X$ be the Hausdorff topological space and let $A$ be both an $\fs$ and $\gd$ subset of $X$. Let also $g\cn X\setminus A\to A\subset X$
be a continuous map. Then the algebraic retraction $\varphi\cn X\to A\subset X$ given by the formula
\begin{equation}\label{fi}
\varphi(x)=\left\{\begin{array}{lll}
x & \textrm{:} & x\in A, \\
g(x) & \textrm{:} & x\in X\setminus A
\end{array} \right.
\end{equation}
belongs to the class $\PC(X,X)\cap\flb(X,X)$. Furthermore, if $X$ is perfectly normal then $\fii\in\fcb(X,X)$. Finally, if $X=\mR$ then
$\fii\in\B_1(X,X)$.
\end{Le}
\begin{proof}
We show first that $\fii\in\PC(X,X)$. Let $i_A\cn A\to A$ be an identity on $A$.
Since $A$ is both an $\fs$ and $\gd$ subset of $X$, there are increasing sequences $(F_n)$ and $(G_n)$ of closed sets such that
$\bigcup_{n}F_n=A$ and $\bigcup_{n}G_n=X\setminus A$. Obviously $F_j\cap G_k=\emptyset$ for every $j,k\in\mN$. Since $g$ is continuous,
the restrictions $i_{A_{\restr F_n}}$ and $g_{\restr G_n}$ are continuous for every $n\in\mN$. Set $H_n=F_n\cup G_n$, for every $n\in\mN$. Obviously
$\bigcup_n{H_n}=X$. Furthermore, the continuity of restrictions
$i_{A_{\restr F_n}}$ and $g_{\restr G_n}$, implies continuity of $\fii$ on every $H_n$. That completes this part of the proof.

To show that $\fii\in\flb(X,X)$ we assume $F$ to be a closed subset of $X$. We have $\fii^{-1}(F)=(A\cap F)\cup g^{-1}(F)$.
Since $A$ is $\fs$ and $F$ closed, $A\cap F$ is $\fs$. Further $g^{-1}(F)$ is a relatively closed subset of an $\fs$-set $X\setminus A$, hence it is again
$\fs$. So, $\fii^{-1}(F)$ is $\fs$.

If $X$ is perfectly normal, then any open set is $\fs$. Hence it follows from the already proved fact that the inverse image
of any open set is $\fs$. Thus $\fii\in\fcb(X,X)$.

To see that $\fii\in\B_1(\mR,\mR)$ it is enough to use the well-known fact that $\fcb(\mR,\mR)=\B_1(\mR,\mR)$ (cf. the Lebesgue-Hausdorff theorem,
\cite[$\S$31, Section IX]{Kur}).
\end{proof}
\section{The main extension results}
The first main theorem considers extending functions from the first Borel class.

\begin{Th}\label{rozsz_fcb}
Let $X$ be a perfectly normal space and let $A$ be its both an $\fs$ and $\gd$ subset. Let also $f\in\fcb(A,\mR)$. Then
there exists a non-negative, preserving unity and linear extension operator $\fii^\star\cn\fcb(A,\mR)\to\fcb(X,\mR)$ such that its restriction to
$\fcb^b(A,\mR)$ is an isometry.
\end{Th}

\begin{proof}
Let $g\cn X\setminus A\to A\subset X$ be a continuous map. Set $\fii^\star(f)=f\circ\fii$, where $\fii$ is an 
algebraic retraction given by (\ref{fi}). Positivity of $\fii^{\star}$ is due to the fact that the class $\fcb(X)$ 
equipped with the order by the coordinates is a linear lattice (see Corollary \ref{krata_liniowa}).
We shall show that $\fii^{\star}(f)\in\fcb(X)$ for every $f\in\fcb(A)$. Let $U$ be an open subset of $\mR$. There is
$$
\left(\fii^{\star}(f)\right)^{-1}(U)=(f\circ \fii)^{-1}(U)=\fii^{-1}\left(f^{-1}(U)\right).
$$
Since $f\in\fcb(A)$, $f^{-1}(U)$ is $\fs$ in $A$. Hence $f^{-1}(U)=\bigcup_{n}K_n$ with every $K_n$ closed in $A$. Thus
$$
f^{-1}(U)=\bigcup_{n}K_n=\bigcup_{n}(A\cap G_n)=A\cap\bigcup_{n}G_n
$$
with every $G_n$ closed in $X$. Therefore
$$
\left(\fii^{\star}(f)\right)^{-1}(U)=\fii^{-1}\left(A\cap\bigcup_{n}G_n\right)=\fii^{-1}(A)\cap\fii^{-1}\left(\bigcup_{n}G_n\right)=
$$
$$
=X\cap\bigcup_{n}\fii^{-1}(G_n)=\bigcup_{n}\fii^{-1}(G_n).
$$
According to Lemma \ref{funkcja fi}, every set $\fii^{-1}(G_n)$ is $\fs$ in $X$. Hence the set
$\left(\fii^{\star}(f)\right)^{-1}(U)=\bigcup_{n}\fii^{-1}(G_n)$ is $\fs$ in $X$ as well.
Thus $\fii^{\star}(f)\in\fcb(X)$.
$\fii^{\star}$ is of course an extension of $f$. Its linearity and preserving of the unity are obvious.
Furthermore, if $f$ is a bounded element of $\fcb(A)$, we have
$$
\|\fii^{\star}(f)\|_X=\sup_{x\in X}|f\left(\fii(x)\right)|=\sup_{a\in A}|f(a)|=\|f\|_A.
$$
\end{proof}

Note that in \cite[Th. 3]{SiegWojtowicz} we have shown, that if $X$ is Hausdorff and $A$ is its both an $\fs$ and $\gd$ subset, then $f\in\PC(A)$ has
a linear extension $\fii^{\star}(f)=f\circ\fii$. In 2005 Kalenda and Spurn\'y proved \cite[Ex. 10]{KalendaSpurny} that there is a bounded function of the
first class $f\cn\mQ\cap[0,1]\to\mR$, which cannot be extended to a function of the first class on $[0,1]$.
From these considerations and by Theorem \ref{Jay_Rog} we obtain the following
\begin{Uw}
Let $A$ be a nonempty subset of a metric space $X$. Let also $f\in\PC(A,\mR)$. If $\ol{f}\in\PC(A_1,\mR)$ extends $f$
then in Theorem \ref{Jay_Rog} we cannot require $A_1$ to be simultanoeusly $\fs$ and $\gd$.
\end{Uw}
\begin{proof}
Let $A=[0,1]\cap\mQ$ and $X=\mR$. $A$ is of course an $\fs$ subset of $X$. Moreover, $\PC(A)=\B_1(A)=\mR^A$. Let $f\cn A\to\mR$. Suppose that there is
an extension $\ol{f}\in\PC(A_1)$ such that $A_1\supset A$ is both an $\fs$ and $\gd$ subset of $\mR$ and $\ol{f}_{\restr A}=f$. Then, by
\cite[Th. 3]{SiegWojtowicz}, the mapping $\fii^{\star}(\ol{f})=\ol{f}\circ\fii$ is a piecewise continuous (and thus Baire one class) extension of $\ol{f}$.
Those arguments would imply that every function $f\in\B_1(A)$ has an extension $\ol{f}\in\B_1(\mR)$, a contradiction.
\end{proof}

The second main theorem deals with extensions of Baire one class functions.
\begin{Th}\label{rozsz_B1}
Let $A$ be both an $\fs$ and $\gd$ subset of $\mR$ and let $f\in\B_1(A,\mR)$.
There exists a non-negative, preserving unity and linear extension operator $\fii^{\star}\cn\B_1(A,\mR)\to\B_1(\mR,\mR)$
such that its restriction to $\B_1^b(A)$ is an isometry.
\end{Th}
\begin{proof}
Let $g\cn\mR\setminus A\to A\subset\mR$ be a continuous map.
Set $\fii^{\star}(f)=f\circ\fii$, where $\fii$ is an algebraic retraction (\ref{fi}).
Since $f\in\B_1(A,\mR)$, there is a sequence $\left(f_n\right)$ of continuous mappings on $A$ such that $f(a)=\lim_{n\to\infty}f_n(a)$, for every
$a\in A$. Since (by Lemma \ref{funkcja fi}) $\fii\in\B_1(A,\mR)$, we have $\fii=\lim_{n\to\infty}\phi_n$ with every $\phi_n$ continuous on $A$.
Obviously every composition $f_n\circ\phi_n$ is continuous on $A$. Let $x\in X$. Since (again by Lemma \ref{funkcja fi}) $\fii\in\PC(A,\mR)$ and 
$\PC(A,\mR)=\B_1^\star(A,\mR)$ (see the Introduction), there is  $n_x\in\mN$ such that $\phi_n(x)=\fii(x)$ for every $n>n_x$. Therefore
$$
\left(f_n\circ\phi_n\right)(x)=f_n\left(\phi_n(x)\right)=f_n\left(\fii(x)\right)\to f\left(\fii(x)\right)
$$
and $f_n\circ\phi_n\xrightarrow{n\to\infty}f\circ\fii=\fii^{\star}\in\B_1(\mR,\mR)$.
Linearity and preserving of the unity are obvious for $\fii^{\star}$.
Furthermore, if $f$ is a bounded element of $\B_1(A,\mR)$, we have
$$
\|\fii^{\star}(f)\|_X=\sup_{x\in X}|f\left(\fii(x)\right)|=\sup_{a\in A}|f(a)|=\|f\|_A.
$$
\end{proof}

\section{Borsuk's non-retract theorem}

The classical Borsuk's non-retract theorem \cite[Theorem 8]{ParkJeong} asserts that
\begin{Th}[Borsuk]
A unit sphere in $\mR^n$ is not a continuous retract of the unit closed ball.
\end{Th}

It turns out that if the word ''continuous retract'' is replaced by ''piecewise continuous retract'',
the Borsuk's claim is no longer true.

\begin{Th}\label{przyk_kula_sfera}
Let $(X,\|\cdot\|)$ be a normed space (or $X$ is simply a closed unit ball in a normed space $Y$) and let $U=X\setminus\{\ol0\}$. 
Let also $A$ be the unit sphere of $X$ and fix $t\in A$. The mapping $\Phi\cn X\to A$ given by the formula
\begin{equation}\label{retrakcja kuli}
\Phi(x)=\left\{\begin{array}{lll}
\frac{x}{\|x\|} & \textrm{:} & x\in U,\\
t & \textrm{:} & x\in X\setminus U
\end{array} \right.
\end{equation}
is piecewise continuous. Hence, the unit sphere is a piecewise contiunous retract of the unit closed ball,
as well as of $Y$.
\end{Th}

\begin{proof}
As elements of the increasing sequence $\left(X_n\right)$ it is enough to take
$X_n=\left\{x\in X:\|x\|\wi\frac{1}{n}\right\}$).
\end{proof}

\begin{Uw}
Let $X$ be a closed unit ball in a normed space and let $U=X\setminus\{\ol0\}$. 
Let also $A$ be the unit sphere of $X$. Lets consider any retraction $\fii\cn X\to A$.
To put an image $\fii(x)$ of any element $x\in U$ into $A$,
we must have $\|\fii(x)\|=1$, so $\fii(x)$ must be a normalization of $x$. On the other hand,
every element $x\in X\setminus U$ must be transformed to $A$ without normalization.
So, the image of $x\in X\setminus U$ must be a point $t\in A$. From these considerations we get
that general form of every retraction $X\to A$ is given by (\ref{retrakcja kuli}).
\end{Uw}

Our last theorem describes a retract of a normed space $X$ onto the open unit ball.
It is an extension of Theorem \ref{przyk_kula_sfera}.
\begin{Th}
Let $(X,\|\cdot\|)$ be at least $1$-dimensional normed space and let $K^0$ be the open unit ball in $X$.
Then the function $\Phi\cn X\to K^0$ given by the formula
$$
\Phi(x)=\left(1-\frac{E(\|x\|)}{\|x\|}\right)\cdot x
$$
is a piecewise continuous retraction from $X$ onto $K^0$.
\end{Th}

\begin{proof}
Note first that
$$
\Phi(x)=\left\{\begin{array}{lll}
x & \textrm{if} & \|x\|<1,\\
0 & \textrm{if} & \|x\|=1.
\end{array} \right.
$$
Since
$$\Phi(x)=(\|x\|-E(\|x\|))\cdot\frac{x}{\|x\|}$$
and 
$$\|\Phi(x)\|=\|x\|-E(\|x\|)<1,$$
 $\Phi$ is a surjection from $X$ onto $K^0$. Set
$$
B_n=\{x\in X:n\mn\|x\|<n+1\}=
$$
$$
=\bigcup_{j=1}^\infty\left\{x\in X:n\mn\|x\|\mn n+1-\frac{1}{j}\right\}=\bigcup_{j=1}^\infty B_{n,j}.
$$
Obviously $X=\bigcup\limits_{n=0}^\infty\bigcup\limits_{j=1}^\infty B_{n,j}$, every set $B_{n,j}$ is closed and every restriction $\Phi_{\restr B_{n,j}}$ is continuous.
\end{proof}

\subsection*{Acknowledgments}
The author would like to express his gratitude to Professor Marek Wójtowicz for his help in preparation of this manuscript.

\end{document}